\documentclass[11pt,reqno]{amsart}
\usepackage{amssymb,latexsym}
\usepackage{graphicx}
\usepackage{url}		

\calclayout
\theoremstyle{plain}
\numberwithin{equation}{section}
\newtheorem{thm}{Theorem}

\newtheorem{lemma}[thm]{Lemma}

\newtheorem*{Euclid}{Euclid's Theorem}
\newtheorem*{Ostrowski}{Ostrowski's Theorem}
\newtheorem*{app}{Approximation Theorem}
\newtheorem*{Euler}{Euler's Theorem}
\newtheorem*{Szemeredi}{Szemer\'edi's Theorem}
\newtheorem*{Euler--Legendre}{Euler--Legendre's Theorem}

\theoremstyle{definition}
\newtheorem*{definition}{Definition}
\newtheorem*{remark}{Remark}

\begin{document}
	\title{Valuations, arithmetic progressions, and prime numbers}
	\author{Shin-ichiro Seki}
	\address{Department of Mathematics\\
		Graduate School of Science\\
		Osaka University\\
		Toyonaka, Osaka\\
		560-0043 Japan}
	\email{shinchan.prime@gmail.com}
	\thanks{The author is supported by the Grant-in-Aid for JSPS Fellows (JP16J01758), The Ministry of Education, Culture, Sports, Science and Technology, Japan.}
	
	\begin{abstract}
		In this short note, we give two proofs of the infinitude of primes via valuation theory and give a new proof of the divergence of the sum of prime  reciprocals by Roth's theorem and Euler--Legendre's theorem for arithmetic progressions.
	\end{abstract}
	
	\maketitle
	
	Many proofs of the following theorem are known after Euclid \cite[Book IX Proposition 20]{H}:
	\begin{Euclid}
		There are infinitely many prime numbers.
	\end{Euclid}
	The author is interested in which mathematical theorems have potential to imply Euclid's theorem without circular arguments. For example, recently, Alpoge \cite{A} discovered that van der Waerden's theorem for arithmetic progressions in Ramsey theory implies Euclid's theorem. As we see in the first section, the approximation theorem in valuation theory of the field of rational numbers also implies Euclid's theorem. In the second section, we see that the idea using an existence of an arithmetic progression is also applicable to prove the divergence of the sum of prime reciprocals. 
	
	\section{The infinitude of primes via valuation theory}
	We cite Neukirch's book \cite{N} for the facts in valuation theory. 
	\begin{Ostrowski}[{\cite[p.~119, (3.7)]{N}}]
		Every non-trivial valuation on the field of rational numbers is equivalent to either the usual absolute value or the $p$-adic valuation for some prime number $p$.
	\end{Ostrowski}
	By this beautiful theorem, we see that the infinitude of primes is equivalent to the infinitude of equivalence classes of non-trivial valuations on the field of rational numbers. In this section, we give two valuation theoretic proofs of Euclid's theorem by the approximation theorem. Note that we don't use Ostrowski's theorem, but we have to consider the infinite place.
	\begin{app}[{\cite[p.~117, (3.4)]{N}}]
		Let $| \ |_1, \dots, | \ |_n$ be pairwise inequivalent non-trivial valuations of the field of rational numbers and let $a_1, \dots, a_n$ be given rational numbers. Then, for every $\varepsilon > 0$, there exists a rational number $q$ such that
		\[
		|q-a_i|_i < \varepsilon
		\]
		for all $i=1, \dots, n$.
	\end{app}
	Let $| \ |_p$ (resp.\ $| \ |_{\infty}$) be the $p$-adic valuation normalized as $|p|_p=p^{-1}$ (resp.\ the usual absolute value) and $\mathbb{Q}_p$ (resp.\ $\mathbb{Q}_{\infty}$) the field of $p$-adic numbers for a prime number $p$ (resp.\ the field of real numbers). In the following two proofs only, we denote $p$ as a prime number or the symbol $\infty$. 
	
	\begin{proof}[Proof of Euclid's theorem by the product formula]
		We assume that there are only finitely many primes. For each $p$, take a rational number $a_p$ such that $|a_p|_p > 1$. By the approximation theorem, we can take a rational number $q$ such that $|q|_p > 1$ for every $p$. Then, $\prod_p|q|_p > 1$ holds. On the other hand, by the product formula (\cite[p.~108, (2.1)]{N}), $\prod_p|q|_p$ must be equal to $1$. This is a contradiction.
	\end{proof}
	\begin{proof}[Proof of Euclid's theorem by the topology of the adele ring] We assume that there are only fin-
		
		\noindent itely many primes. Then, the adele ring $\mathbb{A}_{\mathbb{Q}}$ (\cite[p.~357]{N}) over the field of rational numbers $\mathbb{Q}$ is just the direct product $\prod_p\mathbb{Q}_p$ and the topology coincides with the product topology. Since $\mathbb{Q}$ is dense in $\mathbb{Q}_p$ for each $p$, for the diagonal embedding, $\mathbb{Q}$ is also dense in $\mathbb{A}_{\mathbb{Q}}$ by the approximation theorem. On the other hand, $\mathbb{Q}$ is discrete in $\mathbb{A}_{\mathbb{Q}}$ because of $\{(-1/2, 1/2) \times \prod_{p \neq \infty}\mathbb{Z}_p\}\cap \mathbb{Q} = \{0\}$. Since every discrete subgroup of a Hausdorff topological group is closed, we have that $\mathbb{Q}$ is equal to $\mathbb{A}_{\mathbb{Q}}$. This is clearly impossible.
	\end{proof}
	Furstenberg \cite{F} found a beautiful topological proof of Euclid's theorem. The above second proof also gives a topological proof. 
	
	\section{The divergence of the sum of prime reciprocals via arithmetic progressions}
	
	Euler gave an analytic proof of Euclid's theorem by using the Euler product formula of the Riemann zeta function and the divergence of the harmonic series. Euler proved not only Euclid's theorem, but also the following stronger fact \cite[Theorema 19]{Eu}: 
	\begin{Euler}
		The sum of prime reciprocals diverges$:$
		\[
		\sum_p\frac{1}{p}=\infty.
		\]
	\end{Euler}
	Erd\H os \cite{E} gave another proof of Euler's theorem by combining his combinatorial counting proof of Euclid's theorem and some additional estimate (Lemma \ref{lem2}).
	
	Most of other proofs of Euclid's theorem seem to have no potential to be extended to a proof of Euler's theorem. However, if we use Erd\H os' estimate to ensure that a certain set has positive upper density and use the celebrated theorem by Szemer\'edi instead of using van der Waerden's theorem in the proof of Euclid's theorem by Alpoge \cite{A}, we can give a new proof of Euler's theorem \emph{without the divergence of the harmonic series or quantitative argument}. We recall Szemer\'edi's theorem:
	\begin{definition}[upper density]
		Let $A$ be a set of positive integers. Then, we define \emph{the upper density} $\overline{d}(A)$ of $A$ by
		\[
		\overline{d}(A) := \limsup_{N \to \infty}\frac{\#(A \cap \{1, 2, \dots, N\})}{N}.
		\]
	\end{definition}
	\begin{Szemeredi}
		Let $A$ be a set of positive integers which has positive upper density. Then, $A$ contains an arithmetic progression of length $k$, for every positive integer $k$.
	\end{Szemeredi}
	The case $k=3$ was proved by Roth \cite{R} in 1953 and the case $k=4$ was proved by Szemer\'edi \cite{S1} in 1969. Finally, the general case was established by Szemer\'edi \cite{S2} in 1975. Note that the proofs become much more complicated as $k$ is larger.
	
	Actually, we need not to use Szemer\'edi's theorem, and Roth's theorem is sufficient to deduce Euler's theorem. After the proof by Alpoge, Granville \cite{G} found another way of deducing Euclid's theorem by van der Waerden's theorem. Although we use an existence of an arithmetic progression of sufficiently large length in Alpoge's method, we use only an arithmetic progression of length four in Granville's proof by Fermat's theorem for squares in an arithmetic progression. Furthermore, it is enough in the length-three case if we replace Fermat's theorem with the following Euler--Legendre's theorem \cite[Vol.\ II.\ 572--573]{D}:
	\begin{Euler--Legendre}
		There are no length-three arithmetic progressions whose terms are cubes of positive integers.
	\end{Euler--Legendre}
	Based on the above observation, we give a new proof of the divergence of the sum of prime reciprocals by Roth's theorem and Euler--Legendre's theorem. This is still \emph{overkill}, but we see the power of Roth's theorem. We need the following two lemmas.
	\begin{lemma}[Pigeonhole principle for upper density]
		Let $A$ be a set of positive integers with $\overline{d}(A) > 0$. If $A$ is partitioned into finitely many classes, then there is at least one class which has positive upper density.
		\label{lem1}\end{lemma}
	\begin{proof}
		This is clear by definition.
	\end{proof}
	\begin{lemma}
		Let $p_j$ denote the $j$th prime number and $P_r$ be the set of positive integers which do not have $p_{r+1}, p_{r+2}, \dots$ as their prime factors. We assume that the sum of prime reciprocals converges. Then, there exists some positive integer $r$ such that $P_r$ has positive upper density.
		\label{lem2}\end{lemma}
	This is a rephrasing of Erd\H os' estimate in \cite{E}. 
	\begin{proof}
		Under the assumption, we take a positive integer $r$ as $\sum_{j > r}\frac{1}{p_j} \leq 1/2$ holds. Since positive integers less than or equal to $N$ which are not contained in $P_r$ are divided by at least one of $p_{r+1}, p_{r+2}, \dots$, we have
		\[
		N-\#(P_r \cap \{1, 2, \dots, N\}) \leq \sum_{j > r}\left\lfloor\frac{N}{p_j}\right\rfloor \leq \sum_{j > r}\frac{N}{p_j} \leq \frac{N}{2}
		\]
		for any positive integer $N$. Thus, we have 
		\[
		\#(P_r \cap \{1, 2, \dots, N\}) \geq \frac{N}{2}	
		\]
		and $\overline{d}(P_r) \geq 1/2$.
	\end{proof}
	\begin{proof}[New proof of Euler's theorem]
		We assume that the sum of prime reciprocals converges. Let $r$ and $P_r$ be as in Lemma \ref{lem2}. For a tuple $v \in \{0, 1, 2\}^r$, we define a subset $P_r^{(v)}$ of $P_r$ by 
		\[
		P_r^{(v)} := \{n \in P_r \mid n=p_1^{e_1}\cdots p_r^{e_r}, \ (e_1, \dots, e_r) \equiv v \pmod{3}\}.
		\]
		Then, by Lemma \ref{lem1} and Lemme \ref{lem2}, there exists a tuple $v \in \{0, 1, 2\}^r$ such that the upper density of $P_r^{(v)}$ is positive. Hence, by Roth's theorem, there are positive integers $A$ and $D$ satisfying $A, A+D, A+2D \in P_r^{(v)}$. Let $R$ be the unique cubefree integer in $P_r^{(v)}$. Then, $A$ and $D$ are divided by $R$ and all $a, a+d, a+2d$ are cubes for $a:=A/R$ and $d:=D/R$. This contradicts to Euler--Legendre's theorem. Therefore, the sum of prime reciprocals diverges.
	\end{proof}
	
	\begin{remark}
		Darmon and Merel \cite{DM} proved that there are no non-trivial length-three arithmetic progressions whose terms are $n$th powers for $n \geq 3$.
	\end{remark}

	\section*{Acknowledgment.}
	The author would like to thank Junnosuke Koizumi for letting the author know Alpoge's work \cite{A}. He also would like to thank Kenji Sakugawa, Masataka Ono, Yuta Suzuki, and Toshiki Matsusaka for their valuable comments.

	\medskip
	
	\noindent MSC2010: 11A41, 11B25
	
\end{document}